\documentclass[11pt, a4paper]{amsart}
\usepackage[utf8x]{inputenc}
\usepackage[english]{babel}
\usepackage[toc,page]{appendix}
\usepackage{amsmath,amssymb,mathrsfs,amsthm,dsfont}
\usepackage{graphicx}

\usepackage[lofdepth,lotdepth]{subfig}

\usepackage[a4paper]{geometry}
\usepackage[colorlinks,citecolor=blue,urlcolor=blue,linkcolor=blue,linktocpage=true,pdfproducer={},pdfcreator={}]{hyperref}

\newtheorem{theorem}{Theorem}[section]
\newtheorem{proposition}[theorem]{Proposition}

\newtheorem{lemma}[theorem]{Lemma}

\numberwithin{equation}{section}

\newcommand {\rd}{\mathrm{d}}
\newcommand{\eps}{\varepsilon}

\newcommand{\R}{{\mathbb{R}}}
\newcommand{\N}{{\mathbb{N}}}
\newcommand{\Z}{{\mathbb{Z}}}
\renewcommand{\P}{{\mathbb{P}}}
\newcommand{\E}{{\mathbb{E}}}
\newcommand{\C}{\mathcal{C}}
\newcommand{\A}{\mathcal{A}}

\renewcommand{\L}{\mathcal{L}}
\newcommand{\D}{\mathcal{D}}
\allowdisplaybreaks[1]

\def\restriction#1#2{\mathchoice
              {\setbox1\hbox{${\displaystyle #1}_{\scriptstyle #2}$}
              \restrictionaux{#1}{#2}}
              {\setbox1\hbox{${\textstyle #1}_{\scriptstyle #2}$}
              \restrictionaux{#1}{#2}}
              {\setbox1\hbox{${\scriptstyle #1}_{\scriptscriptstyle #2}$}
              \restrictionaux{#1}{#2}}
              {\setbox1\hbox{${\scriptscriptstyle #1}_{\scriptscriptstyle #2}$}
              \restrictionaux{#1}{#2}}}
\def\restrictionaux#1#2{{#1\,\smash{\vrule height .8\ht1 depth .85\dp1}}_{\,#2}} 

\begin{document}
\title{Equilibration and diffusion for a dynamical Lorentz gas}

\author[E. Soret]{\'{E}milie Soret}

\date{\today}
\address{IEMN UMR CNRS 8520, Avenue Henri Poincaré, 59491 Villeneuve-d'Ascq}
\address{IRCICA CNRS 3024, 50 Avenue  Halley, 59650 Villeneuve-d'Ascq, France.}
\email{emilie.soret@ircica.univ-lille1.fr}

\maketitle 


\begin{abstract} 
We consider a model of a \textit{dynamical} Lorentz gaz: a single particle is moving  in $\R^d$ through an array of fixed and soft 
scatterers each possessing an internal degree of freedom coupled to the particle.  Assuming the initial velocity is sufficiently high 
and modelling the parameters of the scatterers as random variables, we describe the evolution of the kinetic energy of the particle by 
a Markov chain for which each step corresponds to a collision.  We show that the momentum distribution of the particle approaches a 
Maxwell-Boltzmann distribution with effective temperature $T$ such that $k_BT$ corresponds to an average of the scatterers' 
kinetic energy. 
\end{abstract}

\section{Introduction}\label{Sec:Intro}

We study a class of Hamiltonian systems referred to as \textit{dynamical} Lorentz gases and introduced in 
\cite{equilibration,adiabatic}. These models describe the motion of a single particle through an array of independent scatterers, each 
possessing an internal degree of freedom to which the particle is locally coupled. In ~\cite{equilibration}, it is argued that the particle 
momentum distribution $\rho(t)$ will converge, asymptotically in time, to a Maxwell-Boltzmann thermal equilibrium distribution 
characterized by a temperature that is determined by the energy distribution of the individual scatterers . This convergence holds in 
a suitable parameter range corresponding to a weak coupling limit and for an arbitrary distribution $\rho_0(t)$ of sufficiently large 
average mean speed. This result holds even if initially the scatterers are not in thermal equilibrium. In this paper, we provide a 
rigorous proof of this result and identify conditions on the parameter range for which it holds. 
The  original, fully Hamiltonian model consists of a particle-scatterer system which obeys to the following laws of motion :
\begin{equation}\label{eq:motionlaw}
\left\lbrace 
\begin{aligned}
\ddot{q}(t)&=-\alpha \sum_{i\in \Z^d} \eta(Q_i(t))\nabla \sigma(q(t)-r_i),\\
M\ddot{Q}_i(t)+U'(Q_i(t))&=-\tilde{\alpha} \eta '(Q_i(t))\sigma(q(t)-r_i).
\end{aligned}
\right.
\end{equation}
In these equations, $q(t)\in \R^d$, $d\geq 2$ represents the position of the particle at time $t$, and $Q_i\in \R$ is the displacement 
of the internal degree of freedom associated to a scatterer centred at a fixed point $r_i\in \R^d$, $i\in \Z^d$. One may think for example of 
\begin{equation}\label{eq:example_ri}
r_i=i(1)e_1+\cdots +i(d) e_d,
\end{equation}
 with $e_1,\cdots,e_d$ a basis of $\R^d$, and $i(k)$ the $k$-th coordinate of $i$ (see Figure~\ref{fig:traj_part}). The points $r_i$ 
 then lie on a lattice. Another example could be $r_i=i(1)e_1+\cdots +i(d) e_d+\eta_i$, where the $\eta_i$ are chosen in a unit cell of 
 the lattice. In what follows the precise geometry of the scatterers will not play any role, as we will explain in details below. 

\begin{figure}
\centering
\includegraphics[scale=0.15]{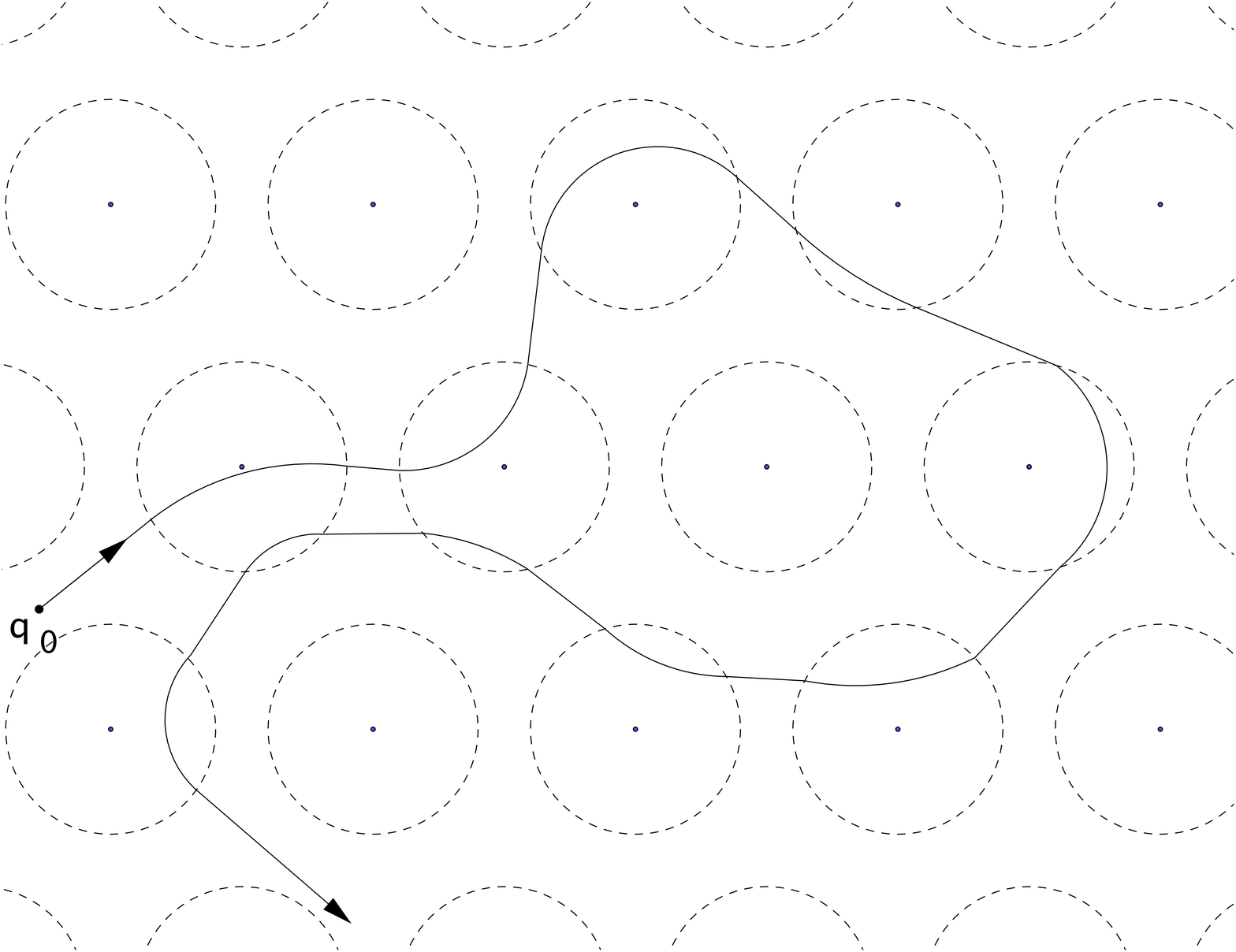}
\caption{Example of path in $\R^2$. Here the points $r_i$  lie on $\Z^2$. The path of 
the particle presents two recollisions.   \label{fig:traj_part}}
\end{figure}

We assume that the scatterers' centres are distributed on $\R^d$ with finite horizon, which means that whatever is the direction of the particle, it cannot move indefinitely without hitting a scatterer. This hypothesis is primordial for our model, indeed if the horizon is not finite, we cannot describe the complete model by a Markovian one as we will do in Section~\ref{Sec:MarkovModel}.

Both the potential $U$ which controls the non-coupled dynamics of the degree of freedom of the scatterers and the coupling function $
\eta$ are supposed to be $\C^\infty(\R)$. More particularly, we assumed that the coupling function $\eta$ is linear and that $U$ 
has a polynomial growth:
\begin{equation}
U(Q)\sim|Q|^r \qquad \text{and}\qquad |\eta(Q)|\sim Q, \qquad r\geq 1.
\end{equation}
Hence, the potential $U$ is confining. 
The form factor $\sigma$ appearing in the interaction term satisfies $\|\sigma\|_\infty\leq 1$. We assume that it is rotationally 
invariant and with compact support in the sphere $B(0,1/2)$. The values $\alpha,\tilde{\alpha}>0$ are coupling constants and $M$ is 
the ratio of the mass associated to the internal degree of freedom over the mass of the particle. All these parameters are 
dimensionless. 

The case $\tilde{\alpha}=0$ is studied in \cite{aguer2010,dbs2015,adblp}. It corresponds to a system in which the passage of the 
particle does not affect the evolution on the environment. We call such  model \textit{inert} (in opposition to \textit{dynamical}). As 
there is no mechanism of dissipation of the energy of the particle, the kinetic energy of the latter grows with the time \cite{dbs2015}. 
This phenomenon is called \textit{stochastic acceleration}. In that case, the particle momentum distribution never approaches an 
equilibrium state, and a fortiori not a thermal equilibrium.

Another interesting case is $\tilde{\alpha}=\alpha$, which is the one here. The particle can be seen as a small perturbation of the 
environment, it leads to a loss of average energy which acts as a source of dynamical friction 
\cite{chandra43a,chandra43b,chandra43c}, and stochastic acceleration does not occur.

In the following, we assume that the two coupling constants are equal, $\alpha=\tilde{\alpha}$. As we consider a linear coupling, the 
law of motion \eqref{eq:motionlaw} becomes
\begin{equation}\label{eq:motionlaw2}
\left\lbrace 
\begin{aligned}
\ddot{q}(t)&=-\alpha \sum_{i\in \Z^d} Q_i(t)\nabla \sigma(q(t)-r_i)\\
M\ddot{Q}_i(t)+U'(Q_i(t))&=-\alpha \sigma(q(t)-r_i).
\end{aligned}
\right.
\end{equation}
 Let $p$ be the momentum of the particle
 and let  $P_i=M\dot{Q}_i$. Then  the equations~\eqref{eq:motionlaw2} are generated by the Hamiltonian
\[H(q,p,Q_i,P_i)=\dfrac{p^2}{2}+\sum_{i\in \Z^d}H_{\textrm{scatt}}(Q_i,P_i)+\alpha \sum_{i\in \Z^d}Q_i\sigma(q-r_i),\]
where $H_\textrm{scatt}$ is the Hamiltonian associated to the dynamics of a single scatterer. The total energy of the system is 
conserved and
\begin{equation}\label{eq:Hscatt}
H_\textrm{scatt}(Q,P)=\dfrac{P^2}{2M	}+U(Q).
\end{equation}
In other words, the internal degree of freedom of each scatterer reacts to the passage of the particle while allowing the energy of the system to be conserved. Obviously, this evolution of the internal degree of freedom does not affect its spatial localisation, nor its interaction area.

Further more, we assume that the particle is always fast which means that it always crosses the interaction region in a time of order $
\|p\|^{-1}$ (see \cite{equilibration}), where $\|\cdot\|$ is for the Euclidean norm in $\R^d$. We will assume moreover that its kinetic 
energy$\|p(t)\|^2/2$ is well above the typical interaction potential $\alpha Q \sigma (q(t)-r)$ that it encounters in any scattering 
event. 

So defined, this dynamical system is too difficult to study with full mathematical rigour. Indeed, the particle can pass through a 
scatterer at least twice: recollisions between the particle and a same scatterer are possible (see Figure~\ref{fig:traj_part}), making it 
necessary to keep track of the evolution of the internal degree of freedom.  In addition, the geometry of the scattering centers $r_i$ 
induces additional difficulties. This means we are dealing with a very complex nonlinear infinite dimensional dynamical system. To 
simplify the problem, we will  follow~\cite{equilibration} and instead describe the trajectory of the particle by a Markov chain, 
eliminating both the complexity induced by recollisions and by the geometry of the family $(r_i)_{i\in \Z^d}$ provided the horizon is 
finite. In this manner, we concentrate on the essential dynamical phenomena induced by the individual scattering events.

Each step of this Markov chain,  described in detail in Section~\ref{Sec:MarkovModel}, corresponds to a collision between the particle 
and a scatterer. We view the states of the scatterer that the particle successively visits as independent and identically distributed 
random variables and fix the distance between two consecutive scatterers met by the particle. In \cite{equilibration}, it was shown, 
for a Markov chain which is a cut version of the one we consider in Section~\ref{Sec:MarkovModel}, that it 
captures the essential features of the behaviour of the original system very well, both in terms of the evolution of the momentum 
distribution of the particle and of its diffusive spatial displacement.

In this paper, our contribution is the obtaining of rigorous results on the asymptotic behaviour of the Markov chain model, in a 
suitable regime for $\alpha$ and $M$, described below, and that corresponds to a weak-coupling limit. We will show that the 
momentum distribution of the particle then converges to a Maxwellian. 

The rest of the paper is organized as follows. In Sec.~\ref{Sec:MarkovModel} we introduce the Markov chain description of the 
particle's momentum (Proposition~\ref{prop:dlmc}) and we identify precisely the parameters and the time scale for which we obtain, 
in Theorem~\ref{thm:averaging}, the weak coupling limit of the Markov chain. Once this limit obtained, we  use in Sec.
\ref{Sec:Averaging},  the Fokker-Plank equation to compute the stationary distribution and we will show that it approaches the 
Maxwellian. The proof of Proposition~\ref{prop:dlmc} and Theorem~\ref{thm:averaging} are given in Sec.\ref{Sec:Appendix}.

\noindent \textbf{Acknowledgement :} This work is in part supported by IRCICA, USR CNRS 3380 and the  Labex CEMPI (ANR-11-
LABX-0007-01). The author thanks  S. De Bièvre and P.E. Parris  for their helpful discussions.

\section{The Markov chain description and weak coupling limit}\label{Sec:MarkovModel}

 To each trajectory $(q(t),p(t))$, solution of the deterministic and Hamiltonian equations of motion~\eqref{eq:motionlaw2}, one can 
 associate a sequence $(t_n,p_n,b_n,r_{i_n},Q_{n},P_{n})_{n\in \N}$. Here, $t_n$ is the instant the particle arrives on the $n$-th 
 scattering region that it will encounter, $p_n=\dot{q}(t_n)$ is its incoming velocity;  $r_{i_n}$ is  the $n$-th scattering center visited 
 by the particle, $Q_{n}$ and $P_{n}$ are the initial states of the scatterer. Last,  the impact parameter $b_n$ models the approach of 
 this scatterer (Figure~\ref{fig:obstacle}) and is, by its definition, orthogonal to the incoming velocity $\|p_n\|$.

\begin{figure}
\centering
\includegraphics[scale=0.5]{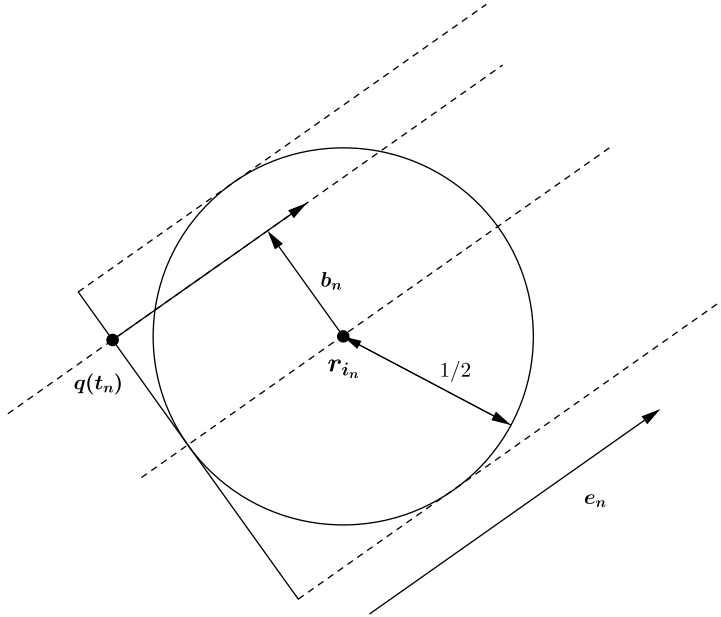}
\caption{Approach of the $n$-th scatterer: the particle arrives at the instant $t_n$ on the $n$-th scatterer it encounter which is 
centered in $r_{i_n}$. At this instant, the particle has position $q(t_n)$ and velocity $p_n=\dot{q}(t_n)$.\label{fig:obstacle}}
\end{figure}

The associated Markov chain is now constructed by first introducing randomness in the parameters $b_n,\, Q_n,\, P_n$, as follows. 
We assume that the sequence $(Q_n,P_n)_n$ is i.i.d with respect to a stationary distribution of the $H_{\textrm{scatt}}$
 \eqref{eq:Hscatt}: 
\begin{equation}\label{eq:stationary_state}
\rho(Q_n,P_n)=\hat{\rho}\left(H_{\textrm{scatt}}(Q_n,P_n)\right)=\rho(Q_n,-P_n),
\end{equation}
where $\rho$ is assumed to be a probability density with compact support.  Similarly, we assume that $(b_n)_n$ is a sequence of i.i.d 
random variables conditioned to be such that, for all $n\in \N$,  the scalar product $b_n\cdot p_n=0$. We denote, as above, by
$p_n$ the momentum of the particle before it encounters, at time $t_n$ and with an impact parameter $b_n$, the $n$-th scatterer. The 
latter is in the initial state $(Q_n,P_n)$. Then,  after scattering, the particle leaves this scatterer with velocity $p_{n+1}$ that it keeps 
until it encounters with a random impact parameter $b_{n+1}$ a $n+1$-th scatterer to which we attribute randomly an initial state $
(Q_{n+1},\, P_{n+1})$.  Explicitly, the momentum change is determined by
\begin{equation}\label{eq:p+R}
p_{n+1}=p_n+R(p_n,b_n,Q_n,P_n), \qquad b_n\cdot p_n=0,
\end{equation}
where
\begin{equation}\label{eq:momentum_transfer}
R(p,b,Q,P)=-\alpha\int_0^{t_+}\rd t\, Q(t)\nabla \sigma (q(t))
\end{equation}
with $t_+$ the instant the particle exits the scatterer and 
with $(q(t),Q(t))$ the unique solution of
\begin{equation}\label{eq:dynamic_unique}
\left\lbrace 
\begin{aligned}
\ddot{q}(t)&=-\alpha Q(t)\nabla \sigma(q(t))\\
M\ddot{Q}(t)+U'(Q(t))&=-\alpha \sigma(q(t)),
\end{aligned}
\right.
\end{equation}
with initial conditions
\[p(0)=p,\quad q(0)=b-\dfrac{p}{2\|p\|}, \quad Q(0)=Q \text{ and } \quad p\cdot b=0.\]
Denoting by $\Delta p_n= p_{n+1}-p_n$ the transfer of momentum, one checks readily (see below also) that this transfer depends 
only on $b_n$, $(Q_n, P_n)$ and $\|p_n\|$.  Finally,  we eliminate the geometry inherent to the initial problem and encoded in the 
positions of the scattering centres $r_i$ by fixing the distance $\ell_*$ the particle travels between two consecutive scatterers.
 With an initial data $(q_0,p_0)$, the position $q_n$ and the momentum $p_n$ of the particle at time $t_n$ are therefore iteratively 
 defined through the relations
\begin{equation}\label{eq:MC1}
\left\lbrace 
\begin{aligned}
p_{n+1}&=p_n+R(p_n,\kappa_n)\\
t_{n+1}&=t_n+\dfrac{\ell_*}{\|p_{n+1}\|}\\
q_{n+1}&= q_n+\dfrac{\ell_*}{\|p_{n+1}\|}p_{n+1},
\end{aligned}
\right.
\end{equation}
where we denote by $\kappa_n$ the triple $(b_n,Q_n,P_n)$.  Note that this is indeed a Markov chain which is  completely determined 
by the initial data $(q_0,p_0)$ and by the sequence of independent and identically distributed random variables $(\kappa_n)_n$.
This provides a simplification of the full deterministic model~\eqref{eq:motionlaw2}.
 
 In the following, for all functions depending on $p$ and other variables such as $\alpha$ and $M$,  we denote its expectation with 
 respect to the random variable $\kappa=(Q,P,b)$ by
\begin{equation}\label{eq:def_expectation}
\overline{f(p,\alpha,M)}=\int\limits_{\substack{\|b\| \leq 1/2 \\ b\cdot p=0}}\dfrac{\rd b}{C_d}\int_{\R\times \R}\rd Q\, \rd P\, \rho(Q,\, P) 
f(p,\, b,\, Q,\, P),
\end{equation}
with $C_d$ the volume of the sphere of radius $1/2$ in $\R^{d-1}$. Moreover, we denote by
\begin{equation}\label{eq:Eetoile}
E_*=\overline{H_{scatt}}=\int \rd Q\, \rd P\,  H_{\mathrm{scatt}}(Q,P)\hat{\rho}\left(H_{\mathrm{scatt}}(Q,P)\right).
\end{equation}
As the distribution $\rho$ is with compact support, the average energy $E_*$ in \eqref{eq:Eetoile} is finite.

In \eqref{eq:MC1}, the first equation, which describes the momentum of the particle after each collision with a scatterer, is 
independent of the two others, thus, we can treat it separately. The transfer of kinetic energy which takes place during a single 
collision for a particle of momentum $p\in \R^d$ is
\begin{equation}\label{eq:energy_transfer1}
\Delta E(\|p\|,\kappa)=\frac{1}{2}\left(\|p+R(p,\kappa)\|^2-\|p\|^2\right).
\end{equation}
Using~\eqref{eq:momentum_transfer}, this energy transfer becomes
\begin{equation}\label{eq:energy_transfer}
\Delta E(\|p\|,\kappa) =\alpha\int_0^{t_+} \rd s\, \dot{Q}(s)\sigma(q(s)).
\end{equation}
 It is the Markov chain for the energy transfer that we shall study here:
\begin{equation}\label{eq:mctrivial}
E_{n+1}=E_n +\Delta E(\sqrt{2E_n}, \kappa_n),
\end{equation}
with $\kappa_n=(Q_n, P_n, b_n)$ as before.

We change the parameters $\left(\alpha,M\right)$ into $\left(\alpha_*,M\right)$ where  \[\alpha_*=\alpha/\sqrt{M}\]
and keep track of the dependence on these parameters in the energy transfer. 
 
The main result of this paper is Theorem~\ref{thm:averaging}. It is a weak coupling limit on the Markov chain~\eqref{eq:mctrivial}. 
The statement is that after an appropriate rescaling of the time variable, the limit of the Markov chain describing the kinetic energy  as $\alpha_*\to 0, \, M\to +\infty$ and $\alpha\to 0$ is a diffusion process which is solution of a well posed martingale problem. But, before giving the full statement of this theorem, 
we need to expand the term of energy transfer of \eqref{eq:mctrivial} in a high energy regime. This is the aim of Proposition~
\ref{prop:dlmc}.

To analyse the weak coupling of the Markov chain~\eqref{eq:mctrivial}, we need the high energy behaviour  of the energy transfer.
 It was shown in~\cite{equilibration} that one can expand $\Delta E$ in inverse powers of $\|p\|$ for large $E$: for all $K\in \N_*$ and 
 for all $(p,\kappa )\in \R^d\times \R \times \R \times \R^d$,
\begin{equation}\label{eq:dl_energy}
\Delta E(\|p\|,\alpha_*,M,\kappa)=\sum_{\ell=0}^K\dfrac{\beta^{(\ell)}(\alpha_*,M,\kappa)}{\|p\|^\ell}+ O\left( \|p\|^{-(K+1)}\right).
\end{equation}
As the time spent by the particle in a scatterer is of order $\|p\|^{-1}$, it appears that, for all $\alpha>0$ and for all $\kappa$,
\[\beta^{(0)}(\alpha_*,M,\kappa)=0.\]
The leading coefficients of this expansion were computed in~\cite{equilibration}. In the following proposition, we provide detailed 
control on the error terms, in particular their behaviour in $M$ and in $\alpha_*$.

Thus, we have the following result on the energy transfer \eqref{eq:energy_transfer}.
\begin{proposition}\label{prop:dlmc}
For all $\alpha_*>0$, $M\gg 1$,  $\kappa=\left(P,Q,b\right)$ and $p\in \R^d$ such that $p\cdot b=0$, we have
\begin{multline}\label{eq:dlmc}
\Delta E(\|p\|,\alpha_*,M,\kappa)=\alpha_*\left(\dfrac{\beta^{(1)}(\kappa)}{\|p\|}+O_0\left(M^{-1/2}\|p\|^{-2}\right)+O\left(M^{-3/2}\|p\|
^{-4}\right)\right)\\
+\alpha_*^2\left(\dfrac{\delta \beta^{(2)}(\kappa)}{\|p\|^2}+O_0\left(M^{1/2}\|p\|^{-3}\right)+\frac{\delta \beta^{(4)}(\kappa)}{\|p\|^4}
+O\left(M^{-1/2}\|p\|^{-5}\right)\right)+o(\alpha_*^2\sqrt{M}),
\end{multline}
where 
\begin{equation}\label{eq:beta1}
\beta^{(1)}(\kappa)=\frac{P}{\sqrt{M}}L_0(\|b\|) \text{ and } \overline{\beta^{(1)}}=0,
\end{equation}
\begin{equation}\label{eq:beta2}
\delta\beta^{(2)}(\kappa)=-\frac{L_0(\|b\|)^2}{2} \text{ and } \overline{\delta \beta^{(2)}}=-\dfrac{\overline{L_0^2}}{2},
\end{equation}
\begin{equation}\label{eq:beta4}
\overline{\delta\beta^{(4)}}=\Sigma_1^2\left(\frac{d-3}{2}+\frac{M\overline{U''}}{\overline{P^2}}\dfrac{\overline{\int_0^1\rd \lambda 
\int_0^\lambda \rd \lambda' K_0(\lambda',\|b\|)\sigma(b+\lambda' e)}}{\overline{L_0^2}}\right),
\end{equation}
with $\Sigma_1^2=\overline{(\beta^{(1)})^2}=\dfrac{\overline{P^2}}{M}\overline{L_0^2}$, and 
\[L_0(\mu,\|b\|)=\int_0^\mu\rd \lambda \, \sigma(b+(\lambda-\frac12)e) \quad \text{and}\quad K_0(\|b\|)=\int_0^1 \rd \lambda \, 
L_0(\lambda,\|b\|).\]

We denote by $O_0(\|p\|^{-k})$ a term of order $O(\|p\|^{-k})$ of zero average with respect to the 
 random variable $\kappa$, see \eqref{eq:def_expectation}.
\end{proposition}

The proof of Proposition~\ref{prop:dlmc} is given in Section~\ref{Sec:Appendix}.

Note that, the coefficients $\beta^{(1)}(\kappa)$, $\delta \beta^{(2)}(\kappa)$ and $\delta \beta^{(4)}(\kappa)$ only depend on $
\kappa$ and not on $(\alpha_*,M)$.
By Proposition~\ref{prop:dlmc}, we can write the transfer of the particle's kinetic energy as 
\begin{multline}\label{eq:cmpn}
\frac{\|p_{n+1}\|^2}{2}=\frac{\|p_n\|^2}{2}+\alpha_*\left(\frac{\beta^{(1)}(\kappa_n)}{\|p_n\|}+O_0(M^{-1/2}\|p_n\|^{-2})+O(M^{-3/2}\|
p_n\|^{-4})\right)\\
+\alpha_*^2\left(\frac{\delta \beta^{(2)}(\kappa_n)}{ \|p_n\|^2}+\frac{\delta \beta^{(4)}(\kappa_n)}{\|p_n\|^4}+O_0(M^{1/3}\|p_n\|^{-3})
+O(M^{-1/2}\|p_n\|^{-5})\right)+o(\alpha_*^2\sqrt{M}).
\end{multline}
We observe that the dominating term of order $\alpha_*^2$ is negative. As the kinetic energy of the particle is a positive quantity, we 
have to avoid that this term leads to a negative energy. Moreover, even if the energy is positive it is not sufficient to describe its 
dynamics by a Markov chain. When the particle is in a low energy state, we do not have any information on its dynamics. It can spent 
a very long time in scatterer, or be trapped in  indefinitely. In this specific case, a description by a Markov chain can not be done. 
The mechanisms of the dynamics at low energy is a complicated subject, so we will introduce a 
critical value $\xi_*\gg 0$  and force the Markov chain to remain higher than this critical 
value using a reflection principle over $\xi_+$.

This cut off in the description of the kinetic energy is a drawback in our model but it is necessary  to have a Markov chain description. We proceed as follow: let $\xi_+\gg  0$ and let $F$ be a function such that for all $x\geq\xi_+$
\begin{multline}\label{eq:defF}
F(x,\alpha_*,M,\kappa)= x+\alpha_*\left(\frac{\beta^{(1)}(\kappa)}{(2x)^{1/2}}+O_0(M^{-1/2}x^{-1})+O(M^{-3/2}x^{-2})\right)\\
+\alpha_*^2\left(\frac{\delta \beta^{(2)}(\kappa)}{ x}+\frac{\delta \beta^{(4)}(\kappa)}{x^2}+O_0(M^{1/3}x^{-3/2})+O(M^{-1/2}
x^{-5/2})\right)+o(\alpha_*^2\sqrt{M}),
\end{multline}
and we will consider the Markov chain $\left(E_n(\alpha_*,M)\right)_n$ with initial condition 
\[\forall \alpha_*>0 \text{ and }\forall M\gg 1, \quad E_0(\alpha_*,M)=\|p(0)\|^2/2\gg \xi_+\] and such that
\begin{equation}\label{eq:finalMC}
\left\lbrace 
\begin{aligned}
E_{n+1}(\alpha_*,M)=F(E_n,\alpha_*,M,\kappa_n) \qquad &\text{if }F(E_n,\alpha_*,M,\kappa_n)\geq \xi_+,\\
E_{n+1}(\alpha_*,M)=2\xi_+-F(E_n,\alpha_*,M,\kappa_n)\qquad &\text{if }F(E_n,\alpha_*,M,\kappa)< \xi_+,
\end{aligned}
\right.
\end{equation}
to shorten  we write $E_n$ instead of $E_n(\alpha_*,M)$ in the function $F$ for a sake of notations.
Hence described, the Markov chain $\left(E_n(\alpha_*,M)\right)$ is a description of the kinetic energy of the particle.

We are now in position to state the weak coupling result we have for the Markov chain $\left(E_n(\alpha_*,M)\right)_n$ defined by \eqref{eq:finalMC}. 
We first construct a family of continuous stochastic processes depending on $(\alpha_*,M)$ and built from the Markov chain $
(E_n(\alpha_*,M))_n$.  For all $n\in \N$, let $\tau_n=\alpha_*^2n$ be the time scale parameter. Then, 
for all $n\in \N$ and for all $\tau \in [\tau_n, \tau_{n+1})$, we define
\[E(\alpha_*,M,\tau)=\dfrac{\tau_{n+1}-\tau}{\alpha_*^2}E_n(\alpha_*,M)+\dfrac{\tau-\tau_n}{\alpha_*^2}E_{n+1}(\alpha_*,M).\]
Thus, for all $\alpha_*>0,\, M\gg 1	$, the process $(E(\alpha_*,M,\tau))_\tau$ is continuous. Note that $\tau$ corresponds to a 
number of collision and not to a time. The following result holds for the family $\left(E(\alpha_*,M,\tau),\tau\in [0,1]
\right)_{(\alpha_*,M)}$.
\begin{theorem}\label{thm:averaging}The family of stochastic processes $\left(E(\alpha_*,M,\tau),\tau\in [0,1]\right)_{(\alpha_*,M)}
$ converges weakly, for $\alpha_*\to 0$, $\alpha\to 0$ and $M\to +\infty$, to the unique solution of the martingale problem 
associated to the operator $(\L,\D_*)$ with
\begin{equation}\label{eq:pbmartalpha}\mathcal{D}_*=\left\lbrace f\in \C^\infty([\xi_+,+\infty)),\, \lim_{x\to \xi_+}f'(x)=0\right\rbrace,
\end{equation}
\begin{equation}\label{eq:pbmartalpha2}\L f(x)=\frac12\frac{\Sigma_1^2}{2x}f''(x)-\left(\dfrac{\overline{\delta\beta^{(2)}_+}}{2x}+
\dfrac{\overline{\delta\beta^{(4)}}}{4x^2}\right)f'(x).
\end{equation}
\end{theorem}
Hence, $\left( \L,\mathcal{D}_*\right)$  is a core for the infinitesimal generator of the process generated by $\L$ with 
reflection over $\xi_+$. We refer to \cite{EthierKurtz,Mandl} for more details on the cores for infinitesimal generator.
The proof of this theorem is given in Sec.~\ref{Sec:Appendix}. 


\section{Stationary distribution}\label{Sec:Averaging}

The probability distribution $\rho(x,\tau)$ of the limiting process satisfies the Fokker-Plank equation with a reflection condition in $
\xi_+$. Denoting by $J$ the probability current (or flow, see \cite{pavliotis2014stochastic}), $\rho$ is solution of 
\begin{equation}\label{eq:FP}
\left\lbrace 
\begin{aligned}
\frac{\partial}{\partial \tau}\rho(x,\tau)&=-\frac{\partial}{\partial x}J(x,\tau)\\
J(\xi_+,\tau)&=0,\qquad \forall \tau >0,
\end{aligned}
\right.
\end{equation}
where 
\begin{align}
\label{eq:J} J(x,\tau)&=-\frac{\partial}{\partial x}\left(D_1(x)\rho(x,\tau)\right)+D_2(x)\rho(x,\tau)\\
\label{eq:D1} D_1(x)&=\frac{1}{2}\frac{\Sigma_1^2}{2x}\\
\label{eq:D2} D_2(x)&=\frac{\overline{\delta\beta^{(2)}}}{2x}+\frac{\overline{\delta \beta^{4}}}{4x^2}.
\end{align}
We are interested in the distribution of the momentum $\rho_{\textrm{eq}}$ when the system has reach is equilibrium state. This 
distribution does not depend on the time any more, i.e $\partial \tau \rho_{\textrm{eq}} \equiv 0$. Hence, it remains to resolve
\begin{equation}\label{eq:FPstation}
J(x)=k
\end{equation}
where $k$ is a constant. As $J(\xi_+)=0$, the stationary distribution $\rho_{eq}$ is solution of 
\begin{equation}\label{eq:stationary}
\forall x\geq \xi_+,\qquad \frac{\partial}{\partial x}\left(D_1(x)\rho_{eq}(x)\right)+D_2(x)\rho_{eq}(x)=0.
\end{equation}
It yields that 
\begin{align}
\nonumber \rho_{eq}(x)&=\frac{\mathcal{N}_0}{D_1(x)}\exp\left(\int_{\xi_+}^x \rd y\, \frac{D_2(y)}{D_1(y)}\right)\\
  &= \mathcal{N}(\xi_+)x^{1+\frac{\overline{\delta \beta^{(4)}}}{\Sigma_1^2}}\exp\left(\dfrac{2\overline{\delta \beta^{(2)}}}{\Sigma_1^2}
  x \right)\\
  &=\mathcal{N}(\xi_+)x^{\frac{d-1}{2}}x^{\frac{C}{k_{scatt}}}\exp\left(-\frac{1}{k_{scatt}} x \right)
\end{align}
with 
\[C=\overline{U''}\frac{\overline{\int_0^1 \rd \lambda \int_0^\lambda \rd \lambda'\, K_0(\lambda')\sigma(b+\lambda'e)}}
{\overline{L_0^2}}\]and with $\mathcal{N}_0$ a normalisation constant and $\mathcal{N}(\xi_+)$ a normalisation constant depending 
on $\xi_+$ and $k_{scatt}^2=\overline{P^2}/M$.
 The change of variable from $E=\|p\|^2/2$ to $\|p\|$, gives,
\begin{align}
\nonumber \hat{\rho}_{\mathrm{eq}}(\|p\|)&=\|p\|\hat{\rho}_{\mathrm{eq}}(\|p\|^2/2)\\
\label{eq:distrilimitep} &=\widetilde{\mathcal{N}}(\xi_+)\|p\|^{d}\|p\|^{2C/k_{scatt}^2}\exp\left(-\dfrac{1}{2 k_{scatt}^2}\|p\|^2\right).
\end{align}
This distribution is associated to the number of collisions for which the final momentum of the particle is $\|p\|$ when the system is in an 
equilibrium state. In order to have a density which describes the state of the particle with respect to the time and not to the number 
of collisions, we remark that the mean time that the particle of velocity $\|p\|$ spends between two consecutive collisions  is $\ell_*/\|
p\|$. As a result, the probability density at the equilibrium $\hat{\rho}_{\mathrm{eq}}(\|p\|)$, describing the time $\rho_{\mathrm{eq}}
(\|p\|)\rd\| p\|$ that the particle spends with a momentum between $\|p\|$ and $\|p\|+ \rd \|p\|$, satisfies the relation
\[\hat{\rho}_{\mathrm{eq}}(\|p\|)=\dfrac{\ell_*}{\|p\|}\rho_{\mathrm{eq}}(\|p\|).\]
Finally, we have
\begin{equation}\label{eq:ditrilimitetemps}
\hat{\rho}_{\mathrm{eq}}(\|p\|)=\mathcal{N}(\xi_+)\|p\|^{d-1+\frac{2C}{k_{scatt}^2}}\exp\left( -\dfrac{\|p\|^2}{2k_{scatt}^2}\right).
\end{equation}
For high values of $k_{scatt}$, this stationary distribution of the particle's momentum \eqref{eq:ditrilimitetemps}, approaches the 
Maxwell-Boltzmann distribution with effective temperature $T$ such that 
\[k_BT=k_{scatt}^2,\]
with $k_B$ the Boltzmann's constant.

We end with a remark on the differences between the results presented here and the ones of \cite{equilibration}. With our
 notations, the Markov chain description of \cite{equilibration} is equivalent to 
\begin{equation}\label{eq:equi_MC}
\frac{\|p_{n+1}\|^2}{2}=\frac{\|p_n\|^2}{2}+ \alpha_*\frac{\beta^{(1)}(\kappa_n)}{\|p_n\|}+\alpha_*^2\left(\frac{\delta \beta^{(2)}
(\kappa)}{ \|p_n\|^4}+\frac{\delta \beta^{(4)}(\kappa)}{\|p_n\|^4}\right).
\end{equation}
Despite some similarities at a first sight with \eqref{eq:cmpn} and between the convergence of these two chains to the same 
Maxwellian distribution, the time scaling of the convergence for \eqref{eq:equi_MC} is not the same that the one needed here.
Indeed, the chain \eqref{eq:cmpn} contains the big $O$ terms which depend on inverse power of $M$. Considering these terms 
implies that a scaling only in $\alpha$ will give us a stationary distribution depending also on $M$, and hence, the convergence to 
the Maxwellian  can not be observed.

\section{Appendix : Proof of Proposition~\ref{prop:dlmc} and of tTheorem~\ref{thm:averaging}}\label{Sec:Appendix}
\subsection{Proof of Proposition~\ref{prop:dlmc}}
First, we expand $q(s)$ in $s$ with the initial condition $q=b-\frac12 e$. Since we always have $s\leq t_+$, which is of order $\|p\|
^{-1}$, any term of order $O(s^k)$ is automatically a term of order $O(\|p\|^{-k})$. We have, 
\[q(s)=q+ps-\alpha Q \nabla K_0(s\|p\|,\|b\|)\|p\|^{-2}-\alpha \frac{P}{M}\nabla K_1(s\|p\|,\|b\|)\|p\|^{-3}+O(s^4),\]
where, for $b\cdot e=0$, we define 
\[L_k(\mu,\|b\|)=\int_0^\mu \rd\lambda \, \sigma(b+(\lambda-\frac12) e))\lambda^k, \quad\text{and }\, K_k(\mu,\|b\|)=\int_0^\mu \rd
\lambda\,  L_k(\lambda,\|b\|).\]
Hence, we proceed as in \cite{equilibration} to write $\Delta E$, defined in \eqref{eq:energy_transfer}, as follow
\begin{multline}\label{eq:DetaE=a+b+c}
\Delta E (\|p\|, \alpha, M,\kappa)=\Delta E_a (\|p\|, \alpha, M,\kappa)+\\
 \Delta  E_b (\|p\|, \alpha, M,\kappa)+\Delta E_c (\|p\|, \alpha, M,\kappa),
\end{multline}
with
\begin{align*}
 \Delta E_a(\|p\|,\alpha,M,\kappa)&=\alpha\int_0^{t_+} \rd s\,  \dot{Q}(s)\sigma(q+ps),\\
\Delta E_b(\|p\|,\alpha,M,\kappa)&=-\dfrac{\alpha^2 Q}{\|p\|^2}\int_0^{t_+} \rd s\, \dot{Q}(s)\nabla K_0(s\|p\|,\|b\|)\cdot \nabla 
\sigma(q+ps),\\
\Delta E_c(\|p\|,\alpha,M,\kappa)&=-\dfrac{\alpha^2 P}{M\|p\|^3}\int_0^{t_+} \rd s\, \dot{Q}(s) \nabla K_1(s\|p\|,\|b\|)\cdot \nabla
\sigma(q+ps).
\end{align*}
As we replace the coupling constant $\alpha$ by $\alpha_*=\alpha/\sqrt{M}$ and as the quantity $P^2/2M$ is the kinetic energy of 
the scatterers, we write, in what follows $P/\sqrt{M}=k_{scatt}$.
Then, with these notations,
 \begin{align}
\label{eq:E_a} \Delta E_a(\|p\|,\alpha_*,M,\kappa)&=\alpha_*\sqrt{M}\int_0^{t_+} \rd s\,  \dot{Q}(s)\sigma(q+ps),\\
\label{eq:E_b}\Delta E_b(\|p\|,\alpha_*,M,\kappa)&=-\dfrac{\alpha_*^2 MQ}{\|p\|^2}\int_0^{t_+} \rd s\, \dot{Q}(s)\nabla K_0(s\|p\|,\|b\|)
\cdot \nabla \sigma(q+ps),\\
\label{eq:E_c}\Delta E_c(\|p\|,\alpha_*,M,\kappa)&=-\dfrac{\alpha_*^2 {k_{scatt}}}{\sqrt{M}\|p\|^3}\int_0^{t_+} \rd s\, \dot{Q}(s) \nabla 
K_1(s\|p\|,\|b\|)\cdot \nabla\sigma(q+ps).
\end{align}

We need to expand in $s$ each of these terms in order to determine the coefficients of order $\|p\|^{-\ell}$ in~\eqref{eq:dl_energy}, 
$\ell \geq 1$. We have already noted that $\beta^{(0)}\equiv 0$ because $t_+$ is of order $\|p\|^{-1}$. Moreover, since $0\leq s\leq 
t_+$ and $t_+$ is of order $\|p\|^{-1}$, to obtain the contributions $\beta_a^{(\ell)},\,\beta_b^{(\ell)},\, \beta_c^{(\ell)}$ from $\Delta 
E_a,\, \Delta E_b$ and $\Delta E_c$ requires that we expand the integrand of the latter.

We first compute $\Delta E_a$. 

We expand $\dot{Q}(s)\sigma(q+ps)$ in $s$ and we identify the terms which do not depend on $\alpha_*$ and those which are of 
order $\alpha_*$. We have, for any $N\geq 1$,
\begin{equation}\label{eq:sum}
\dot{Q}(s)\sigma(q+ps)=\sum_{n=0}^N\sum_{k=0}^n C_n^k\, Q^{(k+1)}(0)s^k\dfrac{\rd^{n-k}}{\rd s^{n-k}}\,\restriction{\sigma(q+ps)}
{s=0} +O(s^{N+1}),
\end{equation}
where, for all $n\leq N$, for all $k\leq n$, $C_n^k$ is the binomial coefficient. To identify the terms of order $\alpha_*$ in $\Delta E_a
$, it remains to identify those that not depend on $\alpha_*$ in~\eqref{eq:sum}. 

Here and in what follows, whenever a function $f$ depends on $\alpha_*$, and possibly on other variables such as $\|p\|$, $M$ and 
$\kappa$, we shall write $f(\alpha_*=0)$ for the values of the function $f$ on the hyperplane $\{\alpha_*=0\}$ and $\delta f:= (f-
f(\alpha_*=0)-o(\alpha_*))/\alpha_*$ for the part of order $\alpha_*$ of the function $f$.

Then, in order to identify the terms which do not depend on $\alpha_*$ in~\eqref{eq:sum}, we have to compute $Q^{(\ell)}(0,
\alpha_*=0)$, $\ell\geq 1$.
One observes that $\Delta E_b$ and $\Delta E_c$ do not contribute to the coefficients $\beta^{(1)}$ and $\beta^{(2)}$, consequently 
$\beta^{(1)}\equiv\beta^{(1)}_a$ and $\beta^{(2)}\equiv\beta^{(2)}_a$. Taking the terms of order $s^0$ in~\eqref{eq:sum}, we easily 
obtain, as $\dot{Q}(0)=k_{scatt}/\sqrt{M}$
\[\beta^{(1)}(\kappa, \alpha_*=0)=k_{scatt}L_0(\|b\|),\]
with $L_0(\|b\|)=\int_0^1\rd\, \lambda \sigma (b+(\lambda-\frac12)e)$. 
As $\beta^{(1)}(\kappa,\alpha_*=0)$ is linear in $P$ and hence of zero average in any stationary distribution $\rho$, $
\overline{\beta^{(1)}(\alpha_*=0)}=0$.
Moreover, $\overline{\beta^{(2)}(\alpha_*=0)}=0$, indeed, this term is obtained by the term of order $s^1$ in~\eqref{eq:sum}, we see 
that they behave as $\ddot{Q}(0,\alpha_*=0)$ which is such that $\overline{\ddot{Q}(0,\alpha_*=0)}=0$, since $\rho$ is stationary 
for the free dynamics of the scatterer generated by $H_\textrm{scatt}$.

Also, we note that $Q^{(3)}(0,\alpha_*=0)$ is linear in $P$, which imply that $\overline{\beta^{(3)}_a(\alpha_*=0)}=0$. 
Finally, for all $k \geq 1$, we easily obtain 
\begin{equation}\label{eq:Q_ell_alpha0}
\sup_{\kappa}\left |Q^{(k+1)}(0,\alpha_*=0)\right | \leq \Phi_k(k_{scatt}) M^{-(k+1)/2},
\end{equation}
with $\Phi_k(k_{scatt})$ a polynomial function of $k_{scatt}$ of degree inferior to $k$.
Let 
\[\Delta E_{a1}=\alpha_* \sqrt{M} \int_0^{t_+}\rd s \, \dot{Q}(s,\alpha_*=0) \sigma(q+ps),\]
be the part of order $\alpha_*$ of $\Delta E_a$. Then, by~\eqref{eq:Q_ell_alpha0}, it yields that
\begin{equation}\label{eq:DeltaEa1}
\Delta E_{a1}(\|p\|,\alpha_*,M,\kappa)=\alpha_* \left(\dfrac{\beta^{(1)}(\kappa,\alpha_*=0)}{\|p\|}+O_0(M^{-1/2}\|p\|^{-2})+O(M^{-3/2}\|
p\|^{-4})\right), 
\end{equation}
Ones observe that as the coefficients $\beta^{(2)}(\kappa,\alpha_*=0)$ and $\beta^{(3)}(\kappa,\alpha_*=0)$ are of zero average, 
they are contained in the term $O_0(M^{-1/2}\|p\|^{-2})$.

Now, we identify the terms of order $\alpha_*$ in~\eqref{eq:sum}, it remains to compute $\delta Q^{(k+1)}(0)$ for $k\geq 0$.  First, 
only $\dot{Q}(0)$ contributes to the terms of order $s^0$ in \eqref{eq:sum} and hence to the coefficient $\beta^{(1)}(\kappa)$. As $
\dot{Q}(0)$ does not depend on $\alpha_*$, $\delta \beta^{(1)}(\kappa) =0$ and 
\[\beta^{(1)}(\kappa)=\beta^{(1)}(\kappa,\alpha_*=0).\]

Thus, as $\dot{Q}(0)$ does not depend on $\alpha_*$, $\delta \dot{Q}(0)=0$. 

Moreover, we have
\[\delta \ddot{Q}(0)= -\frac{1}{M}\restriction{\sigma(q+ps)}{s=0}\]
and, by successive derivations of $Q(s)$, using~\eqref{eq:motionlaw2}, identifying the term on the hyperplane $\{\alpha_*=0\}$ and 
isolating the quantities depending on inverse power of $M$ from $k_{\textrm{scatt}}$, we obtain for $k\geq 2$
\begin{equation}\label{eq:deltaQk}
\overline{\delta Q^{(k+1)}(0)}=\dfrac{1}{M}\|p\|^{k-1}+\sum_{j=0}^{k-3}\dfrac{\tilde{\Phi}_k(k_{scatt})}{M^{(k-j+1)/2}}\|p\|^{j}, \
\end{equation}
where $\tilde{\Phi}_k(k_{scatt})$ is equal to $k_{scatt}^m$ with $0\leq m<k$. Then, combining with \eqref{eq:sum}, and identifying 
the terms of order $s$ which contribute to $\delta \beta^{(2)}_a(\kappa)$, it yields that, factorising by $M^{-1}$ (which goes to $
\alpha_*^2$), 
\[\delta \beta^{(2)}(\kappa)=-M^{-1/2}L_0(\|b\|)^2.\]

Ones can observe that, for all $k\geq 0$,  $\overline{\delta Q^{(k+1)}(0)}$ has no term of order $\|p\|^{-(k-2)}$ and hence, no term of 
order $s^2$ in $\delta\left( \dot{Q}(s)\sigma(q+ps)\right)$ . It yields that \[\delta \beta^{(3)}_a(\kappa)=0.\]
Moreover, controlling the terms of order more that $\alpha_*^2$ and combining \eqref{eq:E_a}, \eqref{eq:DeltaEa1}, \eqref{eq:sum} 
and \eqref{eq:deltaQk}, we obtain
\begin{multline}\label{eq:DeltaEa=}
\Delta E_a(\|p\|,\alpha_*,M,\kappa)= \alpha_*\left(\dfrac{\beta^{(1)}(\kappa)}{\|p\|}+O_0\left(M^{-1/2}\|p\|^{-2}\right)+O\left(M^{-3/2}\|
p\|^{-4}\right)\right)\\
+\alpha_*^2\left(\dfrac{\delta \beta^{(2)}(\kappa)}{\|p\|^2}+\dfrac{\delta \beta^{(4)}_a(\kappa)}{\|p\|^4}+O\left(M^{-3/2}\|p\|^{-5}
\right)\right)+o(\alpha_*^2\sqrt{M}),
\end{multline}
where $\delta \beta^{(4)}_a(\kappa)$ not explicitly computed, but we can have the following estimate by identifying the terms of order 
$s^3$ in \eqref{eq:sum} 
\[\sup_\kappa \left |\delta \beta^{(4)}_a\right | \leq \Phi_3(k_*) M^{-1},\]
and the error term is controlled by $o(\alpha_*^2\sqrt{M})$. 

 As $\Delta E_b$ already depends on $\alpha_*^2$, we just expand its integrand in $s$ and identifies the terms which do not depend 
 on $\alpha_*$. It remains to expand in $s$
\[\dot{Q}(s,\alpha_*=0) \nabla K_0(s\|p\|, \|b\|)\cdot \nabla \sigma(q+ps).\]
We proceed similarly as for $\Delta E_a$ and we obtain 
\begin{multline}\label{eq:sumb}
\dot{Q}(s)\nabla K_0(s\|p\|,\|b\|)\cdot \nabla\sigma(q+ps)\\
=\sum_{n=0}^N\sum_{k=0}^n C_n^k Q^{(k+1)}(0)s^k\dfrac{\rd^{n-k}}{\rd s^{n-k}}\restriction{\nabla K_0(s\|p\|,\|b\|)\cdot \nabla
\sigma(q+ps)}{s=0}+ O(s^{N+1}),
\end{multline}
where, for all $n\leq N$, for all $k\leq n$, $C_{k,n}$ is a constant. Then we identify the terms of order $\alpha_*^2$ in $\Delta E_b$, it 
remains to identify those that do not depend on $\alpha_*$ in~\eqref{eq:sumb}. For that, we need $Q^{(k+1)}(0,\alpha_*=0)$ already 
given in \eqref{eq:Q_ell_alpha0}. As $\Delta E_b$ does not generate any contribution f order $\|p\|^{-1}$ and $\|p\|^{-2}$, combining 
\eqref{eq:E_b}, \eqref{eq:sumb} and \eqref{eq:Q_ell_alpha0} gives us
\begin{equation}\label{eq:DeltaEb=}
\Delta E_b=\alpha_*^2\left(O_0(\sqrt{M}\|p\|^{-3}+ \dfrac{\beta^{(4)}_b(\kappa,\alpha_*=0)}{\|p\|^{-4}}+O\left(M^{-1/2}\|p\|^{-5}\right)
\right)+o(\alpha_*\sqrt{M}).
\end{equation}

Finally, we proceed in the same way for $\Delta E_c$. One observes that $\Delta E_c$ contributes only to the term of order $\|p\|^{-
\ell}$, $\ell \geq 4$. We obtain
\begin{equation}\label{eq:DeltaEc=}
\Delta E_c=\alpha_*^2\left(\dfrac{\beta^{(4)}_c(\kappa,\alpha=0)}{\|p\|^{-4}}+O\left(M^{-1/2}\|p\|^{-5}\right))\right)+o(\alpha_*
\sqrt{M}).
\end{equation}

Moreover, it is shown in~\cite{equilibration} that 
\begin{equation}\label{eq:result_equi}
\overline{\delta \beta^{(4)}_a}+\overline{\beta^{(4)}_b(\alpha=0)}+\overline{\beta^{(4)}_c(\alpha=0)}= \Sigma_1^2 \left( \dfrac{d-3}{2}
+\frac{M\overline{U''}}{\overline{P^2}}\dfrac{\overline{\int_0^1\rd \lambda \int_0^\lambda \rd \lambda' K_0(\lambda',\|b\|)\sigma(b+
\lambda' e)}}{\overline{L_0^2}}\right).
\end{equation}
Finally,  \eqref{eq:DeltaEa=}, \eqref{eq:DeltaEb=}, \eqref{eq:DeltaEc=} and \eqref{eq:result_equi} yield the result.

\subsection{Proof of Theorem~\ref{thm:averaging}}

In this proof, we replace the notation of the Markov chain $\left(E_n(\alpha_*,M)\right)_n$ describing the kinetic energy of the particle 
in \eqref{eq:finalMC} by $\left(\xi_n(\alpha_*,M)\right)_n$.

The scheme of the proof of Theorem~\ref{thm:averaging} is quite classical. We first give and show a result that will allow us to 
conclude to the existence of a converging sub-families. Then we will show that all these converging sub-families have the same limit. 
Thus, we can conclude that it is the whole family that converges to this limiting process. First of all, we need the following lemma.
\begin{lemma}\label{lem:AtoL} For all $\alpha_*>0$ and $M\gg 1$, let $\Pi_{(\alpha_*,M)}$ be the transition function of the Markov 
chain $\left(\xi(\alpha_*,M,\tau_n)\right)_n$ and let  $\left( \Pi_{(\alpha_*,M)}-\mathrm{I}\right)$ be  its generator.
\begin{itemize}
\item[i)]Let $\left( \L,\D_*\right)$ the operator defined in \eqref{eq:pbmartalpha}-\eqref{eq:pbmartalpha2}. For all $f\in \D_*$ and for 
$x\in [\xi_+,+\infty)$
\[\lim_{\substack{\alpha_*\to 0 \\\alpha \to 0\\M\to +\infty}}\left| \L f(x)-\frac{1}{\alpha_*^2}(\Pi_{(\alpha_*,M)}-\mathrm{I})f(x)\right| =0.\]
\item[ii)] The family $\left(\xi(\alpha_*,M,\tau),\, \tau\in [0,1]\right)_{(\alpha_*,M)}$ is pre-compact.
\end{itemize}
\end{lemma}
\begin{proof}
For all $x\in [\xi_+,+\infty)$, consider the following coefficients
\begin{align}
\label{eq:a_and_b} a_{(\alpha_*,M)}(x)&=\int_{[\xi_+,+\infty)}(y-x)^2\Pi_{(\alpha_*,M)}(x,\rd y)\\
\nonumber  b_{(\alpha_*,M)}(x)&=\int_{[\xi_+,+\infty)}(y-x)\Pi_{(\alpha_*,M)}(x,\rd y)
\end{align}
and the operator 
\begin{equation}\label{eq:operator_A}
\A_{(\alpha_*,M)}f(x)=\frac12 a_{(\alpha_*,M)}(x)f''(x)+b_{(\alpha_*,M)}(x)f'(x).
\end{equation}
As the Markov chain $(\xi_n(\alpha_*,M))_n$ is time-homogeneous, the coefficients $a_{(\alpha_*,M)}(x)$ and $b_{(\alpha_*,M)}(x)$ 
are respectively the momentum of order $1$ and $2$ of $\xi(\tau_1)$ conditionally to the event $\xi(0)=x$. 

Therefore, for all $\alpha_*>0$ and for all $M\gg 1$ and for all $x\in [\xi_+,+\infty)$, we have
\begin{multline}\label{eq:Lf-(Pi-I)f_ineq}
\left| \L f(x)-\dfrac{1}{\alpha_*^2}(\Pi_{(\alpha_*,M)}-\mathrm{I}) f(x)\right|\leq \left|\L f(x)-\frac{1}{\alpha_*^2}\A_{(\alpha_*,M)}f(x)
\right| \\
+\frac{1}{\alpha_*^2}\left| \A_{(\alpha_*,M)} f(x)-(\Pi_{(\alpha_*,M)}-\mathrm{I})f(x)\right|.
\end{multline}
First, we will compute the limit for $\alpha_* \to 0$, $\alpha\to 0$ and $M\to +\infty$ for all $x\in [\xi_+,+\infty)$ of $a_{(\alpha_*,M)}
(x)$ and $b_{(\alpha_*,M)}(x)$. Using \eqref{eq:a_and_b}, we have
\begin{multline}\label{eq:a_esperance_pas}
a_{(\alpha_*,M)}(x)=\E\Big[\left(F(\alpha_*,M,x,\kappa)-x\right)^2\mathds{1}_{F(\alpha_*,M,x,\kappa)>\xi_+}\\
+\left(2\xi_+-F(\alpha_*,M,x,\kappa)-x\right)^2\mathds{1}_{F(\alpha_*,M,x,\kappa)\leq\xi_+}\Big]
\end{multline}
and
\begin{multline}\label{eq:b_esperance_pas}
b_{(\alpha_*,M)}(x)=\E\Big[\left(F(\alpha_*,M,x,\kappa)-x\right)\mathds{1}_{F(\alpha_*,M,x,\kappa)>\xi_+}\\
+\left(2\xi_+-F(\alpha_*,M,x,\kappa)-x\right)\mathds{1}_{F(\alpha_*,M,x,\kappa)\leq\xi_+}\Big].
\end{multline}
As the step of the Markov chain is in $\alpha_*$, for all $x>\xi_+$, there exists $\tilde{\alpha}_*$ such that  for all $\alpha_*<
\tilde{\alpha}_*$,
\begin{equation}\label{eq:indicF>xi}
\mathds{1}_{F(\alpha_*,M,x,\kappa)>\xi_+}=1\qquad \text{almost surely}.
\end{equation}
Thus, by Proposition~\ref{prop:dlmc}, for all $x>\xi_+$, there exists $\tilde{\alpha}_*$ such that  for all $\alpha_*<\tilde{\alpha}_*$,
\begin{align*}
a_{(\alpha_*,M)}(x)&=\E\left[\left(F(\alpha_*,M,x,\kappa)-x\right)^2\right]\\
&= \alpha_*^2\Big(\dfrac{\Sigma_1^2}{2x}+O(M^{-1}x^{-2})+O(M^{-3}x^{-4})+O(M^{-1/2}x^{-3/2})\\
& \qquad \qquad \qquad \qquad\qquad \qquad \qquad +O(M^{-3/2}x^{-5/2})+O(M^{-2}x^{-3})\Big)+o(\alpha_*^2\sqrt{M})
\end{align*}
Moreover, by Proposition~\ref{prop:dlmc}, for all $x>\xi_+$, there exists $\tilde{\alpha}_*$ such that  for all $\alpha_*<\tilde{\alpha}_*
$,
\begin{align*}
b_{(\alpha_*,M)}(x)&=\E\left(F(\alpha_*,Mx,\kappa)-x\right)\\
&=\alpha_* O\left(M^{-3/2}x^{-2}\right)+\alpha_*^2\left(\dfrac{\overline{\delta\beta^{(2)}}}{2x}+\dfrac{\overline{\delta\beta^{(4)}}}
{4x^2}+O\left(M^{-1/2}x^{-5/2}\right)\right)+o(\alpha_*^2\sqrt{M}),
\end{align*}
and 
\[b(x)=\lim_{\substack{\alpha_*\to 0 \\\alpha \to 0\\M\to +\infty}}\frac{1}{\alpha_*} b_{(\alpha_*,M)}(x)=\dfrac{\overline{\delta 
\beta^{(2)}}}{2x}+\dfrac{\overline{\delta \beta^{(4)}}}{4x^2}.\]

For all $x>\xi_+$, let 
\[a(x)=\lim_{\substack{\alpha_*\to 0 \\\alpha \to 0\\M\to +\infty}}a_{(\alpha_*,M)}(x)/\alpha_*^2 \text{ and } b(x)=
\lim_{\substack{\alpha_*\to 0 \\\alpha \to 0\\M\to +\infty}}b_{(\alpha_*,M)}(x)/\alpha_*^2,\]
 then, using Proposition~\ref{prop:dlmc}, we have
\begin{equation}\label{eq:a(x)}
a(x)=\frac{\Sigma_1^2}{2x},
\end{equation}
and 
\begin{equation}\label{eq:b(x)}
b(x)=\dfrac{\overline{\delta \beta^{(2)}}}{2x}+\dfrac{\overline{\delta \beta^{(4)}}}{4x^2}
\end{equation}
Now, let $x=\xi_+$, we can easily verify that
\[\lim_{\substack{\alpha_*\to 0 \\ \alpha \to 0\\M\to +\infty}}\frac{1}{\alpha_*^2}a_{(\alpha_*,M)}(\xi_+)=\frac{\Sigma_1^2}{2\xi_+}=a(\xi_
+).\]

However, the drift coefficient does not satisfy 
\[\lim_{\substack{\alpha_*\to 0 \\ \alpha \to 0\\M\to +\infty}}\frac{1}{\alpha_*^2}
b_{(\alpha_*,M)}(\xi_+)=b(\xi_+),\]
 indeed
\begin{multline*}
b_{(\alpha_*,M)}(\xi_+)=\lim_{x\to \xi_+}b_{(\alpha_*,M)}(x)\left(2\P\left(F_{(\alpha_*,M)}(\xi_+,\kappa)\geq 0\right)-1\right)\\
+\xi_+\left(1-\P\left(F_{(\alpha_*,M)}(\xi_+,\kappa)\geq 0\right)\right)
\end{multline*}
and 
\[\lim_{\substack{\alpha_*\to 0 \\ \alpha \to 0\\M\to +\infty}}\frac{1}{\alpha_*^2}b_{(\alpha_*,M)}(\xi_+)\neq \lim_{x\to \xi_+}b(x).\]
Thus, for all $f\in \D_*$, and for all $w\geq \xi_+$, 
\[\lim_{\substack{\alpha_*\to 0 \\\alpha \to 0\\M\to +\infty}}\left| \L f(x)-\frac{1}{\alpha_*^2}\A_{(\alpha_*,M)} f(x)\right|= 0.
\]

It remains to control the second term of \eqref{eq:Lf-(Pi-I)f_ineq}. With a Taylor development of order $2$ and using 
\eqref{eq:a_and_b}, we have that for all $x\geq \xi_+$, there exists $z\in (x,y)$ such that
\begin{equation}\label{eq:pialpha5}\left( \Pi_{(\alpha_*,M)}-\mathrm{I}\right) f(x)\leq \A_{(\alpha_*,M)} f(x)+\frac16\int_{[\xi_+,+
\infty)}f^{(3)}(z)(y-x)^3\Pi_{(\alpha_*,M)}(x,\rd y).
\end{equation}
Moreover, by \eqref{eq:finalMC}, we have
\begin{align*}
\left\vert \int_{[\xi_+,+\infty)}f^{(3)}(z)(y-x)^3\Pi_{(\alpha_*,M)}(x,\rd y)\right\vert &\leq \|f^{(3)}\|_\infty \int_{[\xi_+,+\infty)}|y-x|
^3\Pi_{(\alpha_*,M)}(x,\rd y)\\
\leq\|f^{(3)}\|_\infty & \Big(\E\left( |F_{(\alpha_*,M)}(x,\kappa)-x|^3\mathds{1}_{F_{(\alpha_*,M)}(x,\kappa)\geq \xi_+}\right)\\
+\E\Big(& |2 \xi_++F_{(\alpha_*,M)}(x,\kappa)-x|^3\mathds{1}_{F_{(\alpha_*,M)}(x,\kappa)< \xi_+}\Big)\Big).
\end{align*}
Using \eqref{eq:indicF>xi}, we obtain
\[\left\vert \int_{[\xi_+,+\infty)}f^{(3)}(z)(y-x)^3\Pi_{(\alpha_*,M)}(x,\rd y)\right\vert\leq {\|f^{(3)}\|_\infty} O(\alpha_*^3). \]
Combining the latter with \eqref{eq:pialpha5}, we thus obtain that for all $x\geq \xi_+$, 
\[\left \vert \frac{1}{\alpha_*^2}(\Pi_{(\alpha_*,M)}-\mathrm{I})f(x)-\frac{1}{\alpha_*^2} \A_{(\alpha_*,M)} f(x) \right\vert\ 
\longmapsto_{\substack{\alpha_*\to 0 \\ \alpha \to 0\\M\to +\infty}} 0,\]
which shows the first affirmation.

We, now, show the statement ii). In order to show the pre-compactness of the family $\left( \xi(\alpha_*,M,\tau),\, \tau\in [0,1]
\right)_{(\alpha_*,M)}$ , we first remark that by \eqref{eq:pialpha5} and \eqref{eq:a(x)}, \eqref{eq:b(x)}, for all $f\in \mathcal{D}_*$ 
there exists a constant $\widetilde{C}_f>0$ such that,
\begin{equation}\label{eq:ACfalpha}
\frac{1}{\alpha_*^2}\|(\Pi_{(\alpha_*,M)}-\mathrm{I}) f\|_\infty\leq \widetilde{C}_f.
\end{equation}
It follows that for all $f\in \mathcal{D}_*$,
\begin{align}
\nonumber \E\left(f(\xi(\alpha_*,M,\tau_{n+1}))+\widetilde{C}_f\tau_{n+1}\mid \mathcal{F}^{\xi_{(\alpha_*,M)}}_n
\right)&=f(\xi(\alpha_*,M,\tau_n))+\widetilde{C}_f\tau_n\\
& +(\Pi_{(\alpha_*,M)}-\mathrm{I})f(\xi(\alpha_*,M,\tau_n))+
\widetilde{C}_f\alpha_*^2\\
\label{eq:lastalpha}&\geq f(\xi(\alpha_*,M,\tau_n))+\widetilde{C}_f \tau_n.
\end{align}
Then, the process $\left( f(\xi(\alpha_*,M,\tau_n)+\widetilde{C}_f \tau_n\right)_n$  is a sub-martingale. 
Next, each step of the Markov chain $\left( \xi(\alpha_*,M,\tau_n)\right)_n$ has a size of order $\alpha_*$. So, for all $\delta>0$, 
there exists $\tilde{\alpha_*}$ such that, for all $\alpha_*<\tilde{\alpha_*}$, we have
\[ \P\left(\left\vert \xi(\alpha_*,M,\tau_j)-\xi(\alpha_*,M,\tau_{j-1})\right\vert >\delta\right)=0,\]
and,
\begin{equation}\label{eq:controlstepalpha}
\lim_{\substack{\alpha\to 0\\ \alpha \to 0 \\M\to +\infty}}\sum_{j=1}^{\lfloor \frac{1}{\alpha_*^2}\rfloor}\P\left(\left\vert \xi(\alpha_*,M,
\tau_j)-\xi(\alpha_*,M,\tau_{j-1})\right\vert >\delta\right)=0,
\end{equation}
By the Theorem~$\mathrm{1.4.11}$ of \cite{Stroock}, these two last statements, combined with \eqref{eq:lastalpha}, imply the pre-compactness of the family $\left(\xi(\alpha_*,M,\tau),\, \tau\in [0,1]\right)_{(\alpha_*,M)}$. \qedhere
\end{proof}
Now, we are able to give a proof of Theorem~\ref{thm:averaging}.
\begin{proof}[Proof of Theorem~\ref{thm:averaging}]
statement ii) of Lemma~\ref{lem:AtoL} allows us to establish the existence of decreasing sequences $\{\alpha_{*}(k)\}_k$ with 
positive values such that $\alpha_{*}(k)\to 0$ as $k\to +\infty$ and such that the sub-families $\left( \xi^{\alpha_*(k)}(\tau),\, \tau\in 
[0,1]\right)_{k\in \N}$ converge weakly to limiting processes $\left( \xi(\tau)\right)_{\tau\in [0,1]}$. Consequently, these processes take 
values in $[\xi_+,\, +\infty)$.
Then,  Skorohod's Representation Theorem implies the existence of a probability space $\left(\widetilde{\Omega}, \, 
\widetilde{\mathcal{F}},\, \widetilde{\P}\right)$ and of processes
\[\left(\tilde{\xi}(\alpha_*(k),M,\tau),\, \tau\in [0,1]\right)_{k\in \N}\quad\text{and}\quad \left(\tilde{\xi}(\tau)\right)_{\tau\in [0,1]}\]
respectively with same distribution 
\[\left( \xi(\alpha_*(k),M,\tau),\, \tau\in [0,1]\right)_{k\in \N}\quad \text{and}\quad \left( \xi(\tau)\right)_{\tau\in [0,1]},\]
and such that
\[\sup_{\tau\in [0,1]}\left\vert \tilde{\xi}(\tau)-\tilde{\xi}(\alpha_*(k),M,\tau)\right\vert \longmapsto_{k\to +\infty}0, \qquad \widetilde{\P}
\text{-almost surely.}\]
\begin{lemma}\label{lem:tildealpha}
\begin{itemize} 
For all $f\in \mathcal{D}_*$. We have the two following properties on the Skorohod's space.
\item[i)] Let $0\leq u_1<u_2\leq 1$, then, 
\begin{multline*}
 f\left(\tilde{\xi}(\alpha_*(k),M,\lfloor \frac{u_2}{\alpha_*(k)^2} \rfloor\alpha_*(k)^2 \right) -f\left(\tilde{\xi}(\alpha_*(k),M,\lfloor 
 \frac{u_1}{\alpha_*(k)^2} \rfloor\alpha_*(k)^2  \right)\\
-\sum_{j=\lfloor \frac{u_1}{\alpha_*(k)^2} \rfloor}^{\lfloor \frac{u_2}{\alpha_*(k)^2} \rfloor-1}(\Pi_{(\alpha_*(k),M)}-\mathrm{I})f
\left(\tilde{\xi}(\alpha_*(k),M,\tau_j)\right)
\end{multline*}
converges $\widetilde{\P}$ almost surely as $k\to +\infty$, to
\[ f\left(\tilde{\xi}(u_2\right)-f\left(\tilde{\xi}(u_1\right)-\int_{u_1}^{u_2}\mathcal{L}f\left( \tilde{\xi}(s)\right)ds, \qquad \widetilde{\P}
\text{-p.s.}\]
\item[ii)] The limit process $\left(\widetilde{M}(\tau)\right)_{\tau}$,
\[\widetilde{M}(\tau)=f\left( \tilde{\xi}(\tau)\right)-\int_{0}^{\tau} \mathcal{L}f \left(\tilde{\xi}(s) \right)ds\]
is a $\mathcal{F}_{\tau}^{\tilde{\xi}}$-martingale.
\end{itemize} 
 \end{lemma}
Coming back to the initial probability space, this lemma assures that all processes $\left( \xi(\tau)\right)_{\tau}$ are solutions of the 
martingale problem associated to the operator $\left( \L ,\D_*\right)$ defined in  \eqref{eq:pbmartalpha} and 
\eqref{eq:pbmartalpha2}. We easily check that the coefficients $a(x)$ and $b(x)$ of the operator $\left( \L ,\D_*\right)$ are Lipschitz 
on the interval $[\xi_+,+\infty)$, hence, the martingale problem associated to this operator is well-posed.

As the coefficients of this martingale problem are Lipschitz on $[\xi_+,\, +\infty)$, it is well posed. As a result, the limit processes are 
all equal in distribution. In order to conclude this proof, we show that it is not only the sub-families that converge weakly to the 
solution of this martingale problem but the  whole family $\left( \xi(\alpha_*,M,\tau),\, \tau\in [0,1]\right)_{(\alpha_*,M)}$.
To this end, we will make a reductio ad absurdum. Suppose that there exist $\phi\in \C\left( [0,1],\R\right)$, $\eps>0$  and a 
decreasing sequence $\{\alpha_*(k)\}_k$ of positive numbers which tends to $0$ as $k\to +\infty$ such that, for all $k\in \N$,
\begin{equation}\label{eq:absurdealpha}
\left\vert \E\left( \phi(\xi(\alpha_*(k),M,\cdot)\right)-\E\left( \phi(\xi(\cdot))\right) \right\vert >\eps.
\end{equation}
But the family $\left( \xi(\alpha_*(k),M,\tau),\, \tau\in [0,1]\right)_{k\in \N}$ is still pre-compact. Then, we again can extract a 
converging sub-sequence $\{\alpha_{*}(\varphi(k))\}_k$ such that $\left( \xi(\alpha_{*}(\varphi(k)),M,\tau),\, \tau\in [0,1]\right)_{k\in 
\N}$ converges weakly. According to the foregoing, the limit process of this family is $\left(\xi(\tau)\right)_{\tau\in [0,1]}$. Then, 
\[\lim_{k\to +\infty} \E\left( \phi(\xi(\alpha_{*}(\varphi(k)),M,\cdot)\right)=\E\left( \phi(\xi(\cdot))\right),\]
and  \eqref{eq:absurdealpha} is absurd. Such a function $\phi$ does not exist and it concludes this proof.
\end{proof}
This appendix ends with the proof of Lemma~\ref{lem:tildealpha} used in that of Theorem~\ref{thm:averaging}.
\begin{proof}[Proof of Lemma~\ref{lem:tildealpha}]
i) The limiting processes $\tilde{\xi}$ are continuous in time as limit almost sure of  $\tilde{\xi}(\alpha_*(k),M,\cdot)$ which are 
continuous in time. Then, for all $f\in \mathcal{D}_*$,
\begin{align*}
  \left \vert f\left( \tilde{\xi}(\tau)\right)-f\left( \tilde{\xi}(\alpha_*(k)),M, \lfloor \dfrac{\tau}{\alpha_*(k)^2}\rfloor\alpha_*(k)^2\right)\right
  \vert \leq & ||f'||_{\infty}\Big( \left \vert \tilde{\xi}\left(\tau\right)- \tilde{\xi}\left(\lfloor \dfrac{\tau}{\alpha_*(k)^2}\rfloor
  \alpha_*(k)^2\right)\right \vert \\
+ \Big \vert \tilde{\xi}\Big( \lfloor \dfrac{\tau}{\alpha_*(k)^2}\rfloor\alpha_*(k)^2\Big)&-\tilde{\xi}\Big(\alpha_*(k),M, \lfloor \dfrac{\tau}
{\alpha_*(k)^2}\alpha_*(k)^2\rfloor\Big)\Big \vert \Big)\\
\xrightarrow[k\to +\infty]{} 0 \qquad & \widetilde{\P}-\text{a.s}.
\end{align*}
It remains to control, for all $f\in \mathcal{D}_*$,
\[\int_{u_1}^{u_2}\mathcal{L}f\left( \tilde{\xi}(s)\right)\rd s -\sum_{j=\lfloor \frac{\tau}{\alpha_*(k)^2}\rfloor}^{\lfloor \frac{u_2}
{\alpha_*(k)^2}\rfloor-1}(\Pi_{(\alpha_*(k),M)}-\mathrm{I})f\left( \tilde{\xi}(\alpha_*(k),M,\tau_j)\right)\] 
We have
\begin{align*}
\Big  \vert \int_{u_1}^{u_2}\mathcal{L}f\left( \tilde{\xi}(s)\right)\rd s &-\sum_{j=\lfloor \frac{u_1}{\alpha_*(k)^2}\rfloor}^{\lfloor \frac{u_2}
{\alpha_*(k)^2}\rfloor-1}(\Pi_{(\alpha_*(k),M)}-\mathrm{I})f\left( \tilde{\xi}(\alpha_*(k),M,\tau_j)\right)\Big \vert \\
\leq & \left \vert \int_{u_1}^{u_2}\mathcal{L}f\left( \tilde{\xi}(s)\right)\rd s -\alpha_*(k)^2\sum_{j=\lfloor \frac{u_1}{\alpha_*(k)^2}\rfloor}
^{\lfloor \frac{u_2}{\alpha_*(k)^2}\rfloor-1}\mathcal{L}f\left( \tilde{\xi}(\tau_j)\right)\right \vert \\
+\alpha_*(k)^2 & \left \vert \sum_{j=\lfloor \frac{u_1}{\alpha_*(k)^2}\rfloor}^{\lfloor \frac{u_2}{\alpha_*(k)^2}\rfloor-1}\mathcal{L}f\left( 
\tilde{\xi}(\tau_j)\right)-\sum_{j=\lfloor \frac{u_1}{\alpha_*(k)^2}\rfloor}^{\lfloor \frac{u_2}{\alpha_*(k)^2}\rfloor-1}\mathcal{L}f\left( 
\tilde{\xi}(\alpha_*(k),M,\tau_j)\right) \right \vert \\
+ \alpha_*(k)^2 & \Big \vert \sum_{j=\lfloor \frac{u_1}{\alpha_*(k)^2}\rfloor}^{\lfloor \frac{u_2}{\alpha_*(k)^2}\rfloor-1}\mathcal{L}f\left( 
\tilde{\xi}(\alpha_*(k),M,\tau_j)\right)\\
&- \frac{1}{\alpha_*(k)^2}\sum_{j=\lfloor \frac{u_1}{\alpha_*(k)^2}\rfloor}^{\lfloor \frac{u_2}{\alpha_*(k)^2}\rfloor-1}
(\Pi_{(\alpha_*(k),M)}-\mathrm{I})f\left( \tilde{\xi}(\alpha_*(k),M,\tau_j)\right) \Big \vert.
\end{align*}
The first term is an approximation of the integral and consequently tends to $0$ as $k \to +\infty$. The convergence almost surely of 
$\tilde{\xi}(\alpha_*(k),M,\cdot)$ to $\tilde{\xi}(\cdot)$ implies that the second term tends to $0$ too as $k\to +\infty$. The 
statement i) of Lemma~\ref{lem:AtoL} implies the convergence to $0$ as $k\to +\infty$ of the last term.
 For all $f\in \mathcal{D}_*$, let the process  $\left( \widetilde{S}(\alpha_*(k),M,\tau_n)\right)_{n}$ be defined by 
\[\widetilde{S}(\alpha_*(k),M,\tau_n)=f\left(\tilde{\xi}(\alpha_*(k),M,\tau_n)\right)-\sum_{j=0}^{n-1}(\Pi_{(\alpha_*(k),M)}-\mathrm{I})f
\left(\tilde{\xi}(\alpha_*(k),M,\tau_j)\right).\]
Furthermore, for all $f \in \D_*$, the process$\left( S(\alpha_*(k),M,t_n)\right)_n$ is a $\mathcal{F}_n^{\xi(\alpha_*(k),M)}$-
martingale. Then,
 \begin{multline}\label{eq:produitildeSalpha}
\Big\vert\E\Big[\left(\widetilde{S}(\alpha_*(k),M,\lfloor \frac{u_2}{\alpha_*(k)^2}\rfloor \alpha_*(k)^2)-\widetilde{S}(\alpha_*(k),M,\lfloor 
\frac{u_1}{\alpha_*(k)^2}\rfloor \alpha_*(k)^2)\right)\\
 \phi\left(\tilde{\xi}(\alpha_*(k),M,s_1),\, \cdots,\tilde{\xi}(s_d)\right)\Big]\Big|=0,
 \end{multline}
and for all function  $\phi\in \C_\infty([\xi_+,+\infty)^d)$ with compact support, $d\in \N^*$, and all subdivision
 \[0=s_1<s_2<\dots <s_d=u_1.\]
As we are in the Skorohod' space, the almost surely convergence
\[\phi\left(\tilde{\xi}(\alpha_*(k),M,s_1),\,\dots,\, \tilde{\xi}(\alpha_*(k),M,s_d)\right)\longmapsto_{k\to +\infty} \phi\left(\tilde{\xi}(s_1),\,
\dots,\, \widetilde{\xi}(s_d)\right),\]
is satisfied. Then, by statement i) of this Lemma, the product in the expectation of \eqref{eq:produitildeSalpha} converges almost 
surely to 
\[\left(\widetilde{M}(u_2)-\widetilde{M}(u_1) \right)\phi\left(\tilde{\xi}(s_1),\,\dots,\, \tilde{\xi}(s_d)\right)\]
for all function $\phi\in \C^\infty([\xi_+,+\infty[^d)$ with compact support. Moreover, by \eqref{eq:ACfalpha}, the product in the 
expectation of \eqref{eq:produitildeSalpha} is bounded by 
\[ [2\|f\|_\infty+ \widetilde{C}_f(u_2-u_1)]\|\phi\|_\infty.\]
Using  \eqref{eq:produitildeSalpha} and the dominated convergence theorem, we obtain 
\[\E\left[\left(\widetilde{M}(u_2)-\widetilde{M}(u_1) \right)\phi\left(\tilde{\xi}(s_1),\,\dots,\, \tilde{\xi}(s_d)\right)\right]=0.\]
Therefore process $\left( \widetilde{M}(\tau)\right)_{\tau\in [0,1]}$ is  a  $\mathcal{F}^{\tilde{\xi}}_\tau$-martingale.

\end{proof}
\bibliographystyle{alpha} 
\bibliography{biblio} 
\end{document}